\documentclass[11pt]{amsart}
\usepackage{newlfont,amsmath,mathrsfs,enumerate,amssymb,array,amscd,xcolor,nopageno}

\usepackage{graphicx}
\usepackage[all]{xy}

\theoremstyle{plain}

\newtheorem{thm}[equation]{Theorem}
\newtheorem{prop}[equation]{Proposition}

\newtheorem{lemma}[equation]{Lemma}

\numberwithin{equation}{section}

\newcommand{\Z}{\mathbb Z}

\newcommand{\Aut}{{\mathrm{Aut}}}

\newcommand{\Hom}{{\mathrm{Hom}}}

\newcommand{\rO}{{\mathrm{O}}}

\newcommand{\Sp}{{\mathrm{Sp}}}
\newcommand{\SL}{{\mathrm{SL}}}
\newcommand{\GL}{{\mathrm{GL}}}
\newcommand{\PGL}{{\mathrm{PGL}}}

\newcommand{\SO}{{\mathrm{SO}}}

\def\Z{{\mathbb Z}}

\title {Howe duality for the dual pair $\SL_2(\mathbb R) \times F_{4,1}$ \\ a ping-pong of $K$-types}

\author{Gordan Savin} 
\address{Department of Mathematics, University of Utah, Salt Lake City, UT}
\email{gordan.savin@utah.edu}
\begin{document}

\begin{abstract} 
We prove Howe duality for an exceptional theta correspondence. To that end, 
 we exploit a pair of see-saw identities and relate the $K$-types of corresponding representations.  
\end{abstract} 


\footnote {This work is  supported by the Croatian Science Foundation under the project IP-2022-10-4615 and by a gift No. 946504 from the Simons Foundation.}

\maketitle

\section{Introduction}

\vskip 5pt 
Let $\mathbb O$ be the algebra of Cayley octonions over the field of real numbers $\mathbb R$. 
 Let $J$ be the $27$-dimensional space of $3\times 3$ hermitian symmetric matrices with coefficients in $\mathbb O$.  
 Let $N_J : J \rightarrow \mathbb R$ be the cubic form (the norm of $J$), essentially the determinant of $3\times 3$ matrices. 
For every $e\in J$ such that $N_J(e)\neq 0$ there is a structure of exceptional Jordan algebra on $J$ such that $e$ is the identity of $J$. 
Let $G=\Aut(J,e)$ be the group of automorphisms of the resulting Jordan algebra which is the same as the group of linear transformations of $J$ preserving 
$N_J$ and the point $e$.  It is a simple Lie  group of absolute type $F_4$. See \cite{KMRT} for all of this. 
If we pick $e$ to be 
\[ 
\left(\begin{array}{ccc} 
+1 & & \\
 & +1 &  \\
 & & +1 \\
\end{array} 
\right)
\text{ and } 
\left(\begin{array}{ccc} 
-1 & & \\
 & -1 &  \\
 & & +1 \\
\end{array} 
\right). 
\] 
then $G$ is compact and of split rank one, respectively,  for the two choices of $e$ \cite{OV}.   
The Jordan algebra, by way of Koecher-Tits construction \cite{K67}, gives rise to a simply connected group $G(J)$, of the exceptional type $E_7$ and split rank 3 over $\mathbb R$. 
(The same group for both choices of $e$).  The group $G(J)$ comes along with the dual pair (see \cite{KS15}) 
\[ 
\SL_2(\mathbb R) \times G \subset G(J). 
\] 
These dual pairs are completely analogous to $\SL_2(\mathbb R) \times \rO(p,q)$ in $\Sp_{2n}(\mathbb R)$ where $n=p+q$.  Indeed, if we take $J$ to be 
the space of $n\times n$ symmetric matrices with coefficients in $\mathbb R$, then orthogonal groups are stabilizers of generic points in $J$, and 
 $G(J)$ is $\Sp_{2n}(\mathbb R)$.  

\medskip 

The group $G(J)$ has a minimal (holomorphic) representation that appears as a local component of a global representation \cite{Kim}. In \cite{GS}, the 
representation $\Pi$ was restricted to the dual pair $\SL_2(\mathbb R) \times G$, with $G$ compact, and the following decomposiion was obtained: 
\[ 
\Pi= \bigoplus_{n\geq 0} \delta(2n+12) \otimes E_n. 
\] 
Here $\delta(2n+12)$ is the holomorphic representation of the lowest weight $2n+12$ and $E_n$ is the irreducible representation of $G$ of the highest weight 
$n\varpi_4$ where $\varpi_4$ is the fourth fundamental weight for $F_4$. It is the highest weight of the 26-dimensional irreducible representation of $G$ (the complement 
of the line through $e$ in $J$).  

The goal of this short paper is to study the restriction of $\Pi$ to the dual pair with $G$ non-compact. Let $K$ be the maximal compact subgroup of $G$.  
We emphasize that we do not work with continuous representations of 
non compact groups, but with the corresponding $(\mathfrak g, K)$-modules, where $\mathfrak g$ is the complex Lie algebra of $G$. 
 Thus, if $\pi$ is a $(\mathfrak g, K)$-module of finite length, we define 
\[ 
\Theta(\pi)= (\Pi \otimes \pi^{\vee})_{\mathfrak g}. 
\] 
Here $\pi^{\vee}$ is the contragredient of $\pi$, and the subscript $\mathfrak g$ is saying that we are taking co-invaraints with respect to the action of $\mathfrak g$ on 
the tensor product. We can analogously define $\Theta(\sigma)$ for an $(\mathfrak{sl}_2,\SO(2))$-module $\sigma$ of finite length.  Observe that 
$\Theta(\pi)$ and $\Theta(\sigma)$ are  naturally $(\mathfrak{sl}_2,\SO(2))$ and $(\mathfrak g, K)$-modules, respectively.  In this paper we shall show that 
$\Theta(\pi)$ and $\Theta(\sigma)$ are finite length modules, and that they have unique irreducible quotients, if non-zero.  
 The main input is the structure of lifts of types. More precisely, if $\tau$ is a $K$-type, then $\Theta(\tau)$ is also 
 an $(\mathfrak{sl}_2,\SO(2))$-module that we determine explicitly. Similarly, we have a description of the lift for the $\SO(2)$-types.  
This is done in the last section following an idea of Howe \cite{Ho}. 
With these two pieces of information in hand we can play a game of ping pong with types: if $\sigma \otimes \pi$ is a quotient of $\Pi$ and $\tau$ is a type of $\pi$ then, by a 
see-saw identity,  $\sigma$ must have 
an $\SO(2)$-type determined by $\tau$ and vice versa. More details in the next section where main results are obtained.  
A similar strategy (and the name ping-pong) was used  in \cite{GS} to establish Howe duality for exceptional $p$-adic dual pairs. 

\section{Main results} 
The correspondence with compact $G$ establishes a correspondence of infinitesimal characters in the non-compact case. The reader can consult \cite{Li} for more details on this subject. 
Let us write down the correspondence. Using the standard 
realization of the $F_4$ root system, the infinitesimal character of $E_n$ (the representation with the highest weight $n\varpi_4$) is 
\[ 
\frac{1}{2}(2n+11, 5,3,1).  
\] 
On the other hand, the infinitesimal character of $\delta(2n+12)$ is $2n+11$ which we recognize as the first entry above. 
This means, if $\sigma$ has infinitesimal character $x$, then $\Theta(\sigma)$ has infinitesimal character $\frac{1}{2}( x, 5,3,1)$.  More generally, if $\sigma$ is annihilated 
by an ideal in the center of $U(\mathfrak{sl}_2)$ of finite co-dimension, then $\Theta(\pi)$ is also annihilated by an ideal in the center of $U(\mathfrak g)$ of finite co-dimension. 
Hence,  for $\sigma$ of finite length, in order to prove that $\Theta(\sigma)$ has finite length, it suffices to prove that it is admissible. The same goes for $\Theta(\pi)$.

\smallskip

The maximal compact subgroup of $\SL_2(\mathbb R)$ is $\SO(2)$. Its irreducible representations are one-dimensional characters 
parameterized with integers $n$. Let $(n)$ denote the corresponding one-dimensional representation. 
Since the center of $\SL_2(\mathbb R)$ is also the center of 
the simply connected $G(J)$, only even $n=2m$ characters appear in $\Pi$.  

\smallskip 
 
 The maximal compact subgroup of $G$ is denoted by $K$. It is a simple group of 
 type $B_4$. The group $K$ can be picked to be the intersection of $G$ with the compact form of $G$,  
 where the two groups are the stabilizers of  the two choices for $e$, as in the introduction. 
  Let  $\mathfrak g$ be the complex simple Lie algebra of $G$, and 
 $\mathfrak g=\mathfrak k \oplus \mathfrak p$ the  corresponding Cartan decomposition. 
  Here  $\mathfrak p$ is the 16-dimensional spin representation of $K$.  
  For definiteness, assume that the highest weight of $\mathfrak p$ is 
\[ 
\mu=\frac{1}{2}(1,1,1,1). 
\] 
Let $\lambda=(1,0,0,0)$ be the highest weight of the standard 9-dimensional irreducible representation of $K$. 
Let $\tau(m,n)$ be the irreducible representation of $K$ of the highest weight $m\lambda + n\mu$. Only these representations 
of $K$ appear in the restriction of $\Pi$.  This can be seen by applying the branching rule \cite{Lep} to the representations $E_n$.  
 Let $\Theta(\tau(m,n))$ be the lift of $\tau(m,n)$. 
By Proposition \ref{P:1}, we have 
\[  
\Theta(\tau(m,n))\cong U(\mathfrak{sl}_2) \otimes_{U(\mathfrak{so}(2))} \otimes ( 2m+4). 
\] 
A power of this identity is demonstrated by the following lemma:  

\begin{lemma} \label{L:1} 
 Let $\sigma$ be a finite length $(\mathfrak{sl}_2,\SO(2))$-module. Then 
\[ 
\Hom_K(\Theta(\sigma), \tau(m,n)) \cong \Hom_{\mathfrak{sl}_2} (\Theta(\tau(m,n)), \sigma) \cong \Hom_{\SO(2)}( (2m+4) , \sigma). 
\] 
\end{lemma} 
\begin{proof} The first isomorphism is the see-saw identity, obtained by switching the order of taking $\mathfrak{sl}_2$ and $K$ co-invariants. 
The second isomorphism is the Frobenius reciprocity.  
\end{proof} 

Now we have the following consequence, Santa Claus is coming to town: 

\begin{prop} Let $\sigma$ be a finite length $(\mathfrak{sl}_2,\SO(2))$-module. Then $\Theta(\sigma)\neq 0$ if and only if $\sigma$ has a 
type $2m+4$ for some $m\geq 0$.  $\Theta(\sigma)$ has finite length. If $\sigma$ is irreducible, $\Theta(\sigma)$ has multiplicity free 
$K$-types, consisting of all $\tau(m,n)$ such that $2m+4$ is a type of $\sigma$. 
\end{prop} 
\begin{proof} This is all trivial from the lemma, only finite length of $\Theta(\sigma)$ perhaps merits some explanation. 
It is a combination of admissibility (from lemma) and the fact that 
$\Theta(\sigma)$ is annihilated by an ideal in $Z(\mathfrak g)$ of finite co-dimension.  
\end{proof} 

Now we go in the opposite direction. For a character $2m+4$ of $\SO(2)$ consider  $\Theta(2m+4)$. By Proposition \ref{P:2}, 
 it is a quotient of 
\[ 
U(\mathfrak g)\otimes_{U(\mathfrak k)} F_m 
\] 
where $F_m=\mathbb C$ if $m\leq 0$, otherwise  
\[ 
F_m=\tau(0,0) \oplus \tau(1,0) \oplus \cdots \oplus  \tau(m,0).  
\] 

\begin{lemma} Let $\pi$ be a finite length $(\mathfrak g, K)$-module. Then 
\[ 
\Hom_{\SO(2)}(\Theta(\pi), (2m+4) ) \cong \Hom_{\mathfrak g} (\Theta(2m+4), \sigma) \subseteq  \Hom_{K}(F_m, \pi).  
\]
\end{lemma} 
\begin{proof} The first is the see-saw identity, the second is the Frobenius reciprocity.  
\end{proof} 

A consequence of this lemma is that $\Theta(\pi)$ is admissible and hence of finite length. We are now ready to prove the main result. Keep in mind that, without loss 
of generality, we can assume that $\sigma$ has a type $2m+4$ for some $m\geq 0$. (Otherwise the lift is 0.)

\begin{thm} Let $\sigma$ be an irreducible $(\mathfrak{sl}_2,\SO(2))$-module. Assume that $\sigma$ contains the type $(2m+4)$, for $m\geq 0$, and no smaller types $2n+4$, with $n\geq 0$. 
 (This condition is empty if $m=0$.) Then $\Theta(\sigma)$ has a unique irreducible quotient. It contains the type $\tau(m,0)$ with multiplicity one, and no types $\tau(n,0)$ with $n<m$. 
Conversely, if $\pi$ is an irreducible $(\mathfrak g,K)$-module containing the type $\tau(m,0)$, and 
no smaller such types, then $\Theta(\pi)$ has unique irreducible quotient. It contains the type $(2m+4)$, and no smaller types $2n+4$, with $n\geq 0$. 
\end{thm} 
\begin{proof}  Assume $\pi$ is a quotient of $\Theta(\sigma)$. We do not assume that $\pi$ is irreducible. By Lemma \ref{L:1}, we have the following sequence of inclusions: 
\[ 
\Hom_K(\pi, \tau(m,0)) \subseteq \Hom_K(\Theta(\sigma), \tau(m,0))  \cong \Hom_{\mathfrak{sl}_2} (\Theta(\tau(m,0)), \sigma) \cong \Hom_{\SO(2)} ((2m+4) , \sigma). 
\] 
Observe that we can ran this sequence with any $2n+4$ in place of $2m+4$. If $n<m$, by the assumption, the last space is trivial, hence $\tau(n,0)$ is not a type of 
$\pi$. We shall use this in a moment. 
Since $\sigma$ is a quotient of $\Theta(\pi)$ we have the second sequence of inclusions (note that we are starting with the space of the same dimension as 
$\Hom_{\SO(2)} ((2m+4) , \sigma)$: 
\[ 
\Hom_{\SO(2)}(\sigma, (2m+4)) \subseteq \Hom_{\SO(2)}(\Theta(\pi), (2m+4))  \cong \Hom_{\mathfrak g} (\Theta(2m+4), \pi) \subseteq  \Hom_{K} (F_m , \pi). 
\] 
Since, $\pi$ has no type $\tau(n,0)$ with $n<m$, we ended with  $\Hom_{K} (\tau(m,0) , \pi)$, which has  the same dimension as $\Hom_K(\pi, \tau(m,0))$. 
Thus all inclusions in the two sequences  are isomorphisms, and thus all spaces are one-dimensional, since $\Hom_{\SO(2)} ((2m+4) , \sigma)$ is one-dimensional. 
However, we did not assume that $\pi$ is irreducible. If $\pi=\pi_1\oplus \pi_2$ then, 
\[ 
1+1 = \dim \Hom_{K}(\pi_1,   \tau(m,0)) + \dim\Hom_{K}(\pi_2,  \tau(m,0)) = \dim  \Hom_{K}(\pi, \tau(m,0))=1
\] 
a contradiction. Thus $\Theta(\sigma)$ has a unique irreducible quotient. It contains $\tau(m,0)$, with multiplicity one.  The other direction is proved in the same way... 
\end{proof}

\section{Computing lifts of types} 

In this section we verify the expressions for $\Theta(\tau(m,n))$ and $\Theta(2m+4)$ used in the proof of the main result. 
The idea comes from \cite{Ho} and uses the following  see-saw diagram in $G(J)$.  
Here $H$ is a simply connected, hermitian symmetric group of absolute type $E_6$.  The centralizer of $H$ is  $\SO(2)$, as the picture indicates. 
The centralizer of $K$ is  $\SL_2\times\SL_2$, where the centralizer of $G$ sits diagonally.

 
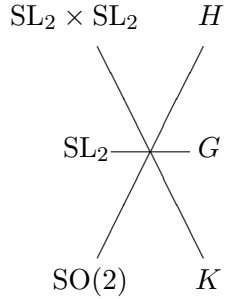
\begin{picture}(100,130)(-130,15) 

\put(20,24){$\SO(2)$} 
\put(24,74){$\SL_2$} 
\put(04,124){$\SL_2\times \SL_2$}

\put(74,24){$K$}
\put(75,74){$G$}
\put(75,124){$H$}

\put(37,36){\line(1,2){40}}

\put(42,76){\line(1,0){30}} 

\put(37,116){\line(1,-2){40}}

\end{picture} 

A word of caution here. If we pick a different $\SL_2$  in  $\SL_2\times \SL_2$, the one consisting of all $(g, hgh^{-1})$  where 
$h=\left(\begin{smallmatrix} 1 & 0 \\ 0 & -1 \end{smallmatrix}\right)$, then $G$ and $H$ in the above are replaced by their compact forms.  
In other words, it is important how we identify groups isomorphic to $\SL_2(\mathbb R)$. 

\smallskip 

Let $(e,h,f)$ be an $\mathfrak {sl}_2$-triple such that $\mathbb C \cdot h$ is the Lie algebra of $\SO(2)$.  
For an integer $n>0$, let $\delta(n)$ be the irreducible lowest weight $n$ module. Let $v_n$ be a non-zero lowest weight vector. 
Let $\bar{\delta}(m)$  be the complex conjugate of $\delta(m)$.  It is the irreducible highest weight $-m$ module. 
 Observe that there is a natural map 
\[
U(\mathfrak {sl}_2) \otimes_{U(\mathfrak {so}(2))} (n-m) \rightarrow  \delta(n) \otimes \bar{\delta}(m)
\] 
where $1\in \mathbb C \cong (n-m)$ is mapped to $v_n\otimes v_m$. Since $\delta(n)$ is a free $\mathbb C[e]$-module generated by $v_n$ and 
$\bar{\delta}(m)$ is a free $\mathbb C[f]$-module generated by $v_m$, the above map is easily checked to be an isomorphism.  

\begin{prop} \label{P:1}
Let $\tau(m,n)$ be the irreducible $K$-type of the highest weight $\frac{1}{2}( 2m+n,n,n,n)$.  Then 
\[ 
\Theta(\tau(m,n))\cong U(\mathfrak{sl}_2) \otimes_{U(\mathfrak{so}(2))} \otimes ( 2m+4). 
\] 
\end{prop} 
\begin{proof} Since the centralizer of $K$ is $\SL_2 \times \SL_2$, $\Theta(\tau(m,n))$ is naturally an $\SL_2 \times \SL_2$-module. 
By \cite[Proposition 3.3.3]{Shan} (careful with $\SL_2$'s) we have 
\[ 
\Theta(\tau(m,n))\cong \delta(2m+n+8)  \otimes \bar{\delta}(n+4). 
\] 
In view of the discussion above,  and $(2m+n+8)-(n+4)=2m+4$, the proposition follows. 

\end{proof} 

It remains to discuss $\Theta(2m+4)$. Let $L$ be a maximal compact subgroup of $H$. We can assume that $K\subset L$. Let $\mathfrak h$ and $\mathfrak l$ be the 
complex Lie algebras of $H$ and $L$.  Since $H/L$ is a hermitian symmetric space,  $\mathfrak l$ is a Levi subalgebra such that $[\mathfrak l, \mathfrak l]$ is a simple Lie algebra of type $D_5$. 
We have a Cartan decomposition 
\[ 
\mathfrak h = \bar{\mathfrak u} + \mathfrak l + \mathfrak u
\] 
such that $\mathfrak q= \mathfrak l + \mathfrak u$ is a parabolic subalgebra. If $F$ is a finite-dimensional $\mathfrak l$-modue, we can define a highest weight module 
\[ 
U(\mathfrak h) \otimes_{U(\mathfrak q)} F \cong U(\bar{\mathfrak u}) \otimes F. 
\] 
We now restrict this module to $\mathfrak g=\mathfrak k + \mathfrak p$. Recall that $\mathfrak p$ is 16-dimensional spin-module. On the other hand, $\bar{\mathfrak u}$ and $\mathfrak u$ 
are two 16-dimensional spin modules for $[\mathfrak l, \mathfrak l]$, the simple algebra of type $D_5$. Hence  $\mathfrak p$ must embed diagonally into $\bar{\mathfrak u} + \mathfrak u$. 
Now it is not difficult to check that the natural map 
\[ 
U(\mathfrak g) \otimes_{U(\mathfrak k)} F \rightarrow U(\mathfrak h) \otimes_{U(\mathfrak q)} F
\] 
given by the idenity on $1\otimes F$ is an isomorphism.  We are ready to prove the following: 

\begin{prop} \label{P:2}
Let $m$ be an integer. Then $\Theta(2m+4)$ is a quotient of $U(\mathfrak g)\otimes_{U(\mathfrak k)} F_m$  where 
$F_m=\mathbb C$ if $m\leq 0$, otherwise 
\[ 
F_m=\tau(0,0) \oplus \tau(1,0) \oplus \cdots \oplus \tau(m,0). 
\] 
\end{prop} 
\begin{proof}  $\Theta(2m+4)$ is an $(\mathfrak h, L)$-module, determined in \cite[Section 6]{Pavle}. It is a quotient of the Verma module $U(\mathfrak h) \otimes_{U(\mathfrak q)} F_m$ 
where $F_m$ is a one dimensional representation of $L$ if $m\leq 0$. Otherwise $F_m$, restricted to $[L,L]$, is irreducible with the 
highest weight $(m,0,0,0,0)$.  The restriction of this representation to $K$ is the claimed sum. 
\end{proof}


\end{document}